\documentclass[12pt]{amsart}
\usepackage[utf8]{inputenc}
\usepackage{amsmath}
\usepackage{amsfonts}
\usepackage{amssymb}
\usepackage{graphicx}
\usepackage{tikz,comment}

\newtheorem{theorem}{Theorem}
\theoremstyle{plain}

\newtheorem{corollary}[theorem]{Corollary}

\newtheorem{proposition}[theorem]{Proposition}

\numberwithin{equation}{section}

\newcommand{\rn}{\mathbb{R}^n}
\newcommand{\rnn}{\mathbb{R}^{2n}}
\newcommand{\R}{\mathbb{R}}
\begin{document}

\title[Strichartz estimates]{Strichartz estimates for the metaplectic representation}
\author{Alessandra Cauli, Fabio Nicola \and Anita Tabacco}
\address{Dipartimento di Scienze Matematiche, Politecnico di Torino, Corso
Duca degli Abruzzi 24, 10129 Torino, Italy}
\email{alessandra.cauli@polito.it}
\email{fabio.nicola@polito.it}
\email{anita.tabacco@polito.it}
\date{}

\begin{abstract} Strichartz estimates are a manifestation of a dispersion phenomenon, exhibited by certain partial differential equations, which is detected by suitable Lebesgue space norms. In most cases the evolution propagator $U(t)$ is a one parameter group of unitary operators. Motivated by the importance of decay estimates in group representation theory and ergodic theory, Strichartz-type estimates seem worth investigating when $U(t)$ is replaced by a unitary representation of a non-compact Lie group, the group element playing the role of time. Since the Schr\"odinger group is a subgroup of the metaplectc group, the case of the metaplectic or oscillatory representation is of special interest in this connection.  We prove uniform weak-type sharp estimates for matrix coefficients and Strichartz estimates for that representation. The crucial point is the choice of function spaces able to detect such a dispersive effect, which in general will depend on the given group action. The relevant function spaces here turn out to be the so-called modulation spaces from Time-frequency Analysis in Euclidean space, and Lebesgue spaces with respect to Haar measure on the metaplectic group. The proofs make use in an essential way of the covariance of the Wigner distribution with respect to the metaplectic representation. 
\end{abstract}

\maketitle

\maketitle
\section{Introduction and statement of the main results}

Strichartz estimates represent one of the main research theme in modern Harmonic Analysis and Partial Differential Equations. The literature in this connection is growing incredibly fast, and new results are often applied to wellposedness and scattering of nonlinear PDEs, see \cite{Tao} and the references therein.\par
 Maybe the simplest case is given by the free Schr\"odinger equation in $\mathbb{R}^n$. The corresponding propagator $U(t)=e^{it\Delta}$ is easily proved to satisfy the so-called {\it dispersive estimate}:
\[\|U(t)\psi\|_{L^\infty}\lesssim |t|^{-n/2}\|\psi\|_{L^1}.
\]
One then deduces mixed-norm estimates, known as {\it Strichartz estimates}, which read
 \[
 \|U(t)\psi\|_{L^q(\mathbb{R};L^r(\mathbb{R}^n))}\lesssim\|\psi\|_{L^2(\mathbb{R}^n)}\] 
 for $\frac{2}{q}+\frac{n}{r}=\frac{n}{2}$, $2\leq q$, $r\leq \infty$ and $(q,r,n)\not=(2,\infty,2)$.
\par
Strichartz estimates are a manifestation of two effects: compared with the basic $L^2$-conservation law, corresponding to the pair $q=\infty,r=2$, the other pairs express\medskip
\begin{itemize}
\item a gain (loss) of space (time) local regularity, \smallskip
\item a gain (loss) of time (space) decay at infinity. 
\end{itemize}\medskip
Dispersive and Strichartz estimates hold, for different ranges of exponents, and possibly with a loss of derivatives, for several classes of equations, even on manifolds, homogeneous spaces, etc. In general, the evolution propagator $U(t)$ is a strongly continuous unitary representation of the abelian group $\R$. Now, for a non-compact abelian group $G$, the {\it irreducible} unitary representations are one-dimensional and their matrix coefficients are just (multiples of) the group characters, with no decay at all.
The above decay is in part due to a lack of ``coherence'' of the irreducible components of $U(t)$: frequency components move in different directions and, in some cases, with different speeds. \par 
Motivated by the importance of decay estimates in representation theory and ergodic theory (see e.g.\ \cite{howe2,moore} and the references therein), Strichartz-type estimates seem worth investigating for strongly continuous unitary representations
$
\mu:G\to \mathcal{U}(H)
$
of a non-compact locally compact Hausdorff group $G$, where $H$ is a Hilbert space. The representation $\mu(g)$ plays now the role of the above propagator $U(t)$. 
Generally speaking, we are interested in estimates of the type 
\begin{equation}\label{coeff0}
||\mu(g) \psi\|_{L^q(G; X_\theta)}\lesssim\|\psi\|_{H}
\end{equation}
for some scale of Banach spaces $X_\theta$, valid for a range of pairs $(q,\theta)$. \par
In this note we develop this idea for the metaplectic group $G=Mp(n,\R)$, that is the double covering of the symplectic group $Sp(n,\R)$, and the corresponding metaplectic, or oscillatory, representation, first constructed by Segal and Shale \cite{26,27} in the framework
of quantum mechanics (see also van Hove \cite{38}) and by Weil \cite{39} in number theory. 
This is a strongly continuous unitary representation of $Mp(n,\R)$ in $L^2(\rn)$, which turns out to be faithful, so that we can think of $Mp(n,\R)$ as a subgroup of $\mathcal{U}(L^2(\rn))$, and the representation given just by the inclusion. Following \cite{DeGosson} we will therefore denote by $\widehat{S}$ a metaplectic operator and by $S=\pi(\widehat{S})\in Sp(n,\R)$ its projection in the symplectic group (the construction of the metaplectic representation is briefly recalled in Section \ref{sec2} below).\par
Now, it turns out that the operator $e^{it\Delta}$ is a particular metaplectic operator, so that a natural candidate for the spaces $X_\theta$ in \eqref{coeff0} would seem to be the Lebesgue spaces. However, the Fourier transform is itself a metaplectic operator, and therefore we should actually look for spaces invariant with respect to the action of the Fourier transform. $U(n)$-invariance (see Section \ref{sec4}) finally suggests, as right function spaces, the modulation spaces $M^p$, widely used in Time-frequency Analysis \cite{DeGosson,Grochenig}.\par In short, for a given Schwartz function $\varphi\in\mathcal{S}(\rn)\setminus\{0\}$, consider the time-frequency shifts $\varphi_{z}(y)=e^{i\xi\cdot y}\varphi(y-x)$, $z=(x,\xi)\in\rn\times\rn$. Then for $1\leq p\leq\infty$ we define the $M^p$ norm of $\psi\in\mathcal{S}'(\rn)$,  as 
\[
\|\psi\|_{M^p}=\Big(\int_{\rnn} | \langle \psi,\varphi_z\rangle |^p\, dz\Big)^{1/p}
\]
(with obvious changes when $p=\infty$). Different windows $\varphi$ give equivalent norms. 
We have $\mathcal{S}(\rn)\subset M^p\subset \mathcal{S}'(\rn)$ for every $1\leq p\leq \infty$, $M^2=L^2(\rn)$, $M^p\subset M^q$ if and only if $p\leq q$,  $(M^p)'=M^{p'}$ if $p<\infty$. Modulation space norms measure the phase space concentration of a function; roughly speaking we can think of a function in $M^p$ as a function having $L^p$ decay at infinity and $F L^p$ local regularity. 
Let us also observe that modulation spaces have been recently applied in PDEs by several authors, see e.g. \cite{cordero3,cordero1,cordero2,ruzhansky,wang} and the references therein (some of their properties are collected in Section \ref{sec2}). \par\medskip
We follow the usual pattern, namely we begin with a dispersive-type estimate. 
\begin{theorem}[Dispersive estimate]\label{mainthm}
The following estimate holds: 
\begin{equation}\label{dispersiva}
\|\widehat{S}\psi\|_{M^{\infty}}\lesssim (\lambda_{1}(S)\ldots\lambda_{n}(S))^{-1/2}\|\psi\|_{M^1}
\end{equation}
for $\widehat{S}\in Mp(n,\mathbb{R}),\ \psi\in\mathcal{S}(\rn)$, where $\lambda_{1}(S),\ldots,\lambda_{n}(S)$ are the singular values $\geq 1$ of $S=\pi(\widehat{S})\in Sp(n,\mathbb{R})$.
\end{theorem}
The result is sharp as far as the decay is concerned (see Section \ref{sec4}).\par
As a consequence we can obtain the following estimates on matrix coefficients.
\begin{corollary}[Uniform weak-type estimate for matrix coefficients]\label{mainthm2} Let $G=Mp(n,\R)$ with the Haar measure. The following estimate holds: 
\begin{equation}\label{coeff}
\|\langle \widehat{S}\varphi_1,\varphi_2\rangle\|_{L^{4n,\infty}(G)}\lesssim\|\varphi_1\|_{M^1}\|\varphi_2\|_{M^1},
\end{equation}
for 
$\varphi_1,\varphi_2\in\mathcal{S}(\rn)$. 
\end{corollary}
Here $L^{4n,\infty}$ is the weak-type $L^{4n}$ space on $G=Mp(n,\mathbb{R})$. \par
Corollary \ref{mainthm2} refines a result by Howe \cite{howe}, who proved that for fixed $\varphi_1,\varphi_2\in \mathcal{S}(\mathbb{R}^n)$ the matrix coefficient in \eqref{coeff} is in $L^{4n+\epsilon}$ for every $\epsilon>0$ but in general not in $L^{4n}$. In fact, estimates for matrix coefficients have a long tradition in representation theory, see for example \cite{cowling2,cowling,howe,howe2,oh} and the references therein. Usually, dealing with a unitary representation of a group $G$ in a Hilbert space $H$, one takes $\varphi_1,\varphi_2$ in $K$-invariant finite dimensional subspaces of $H$, $K\subset G$ being a maximal compact subgroup, and the constants in the estimates will depend on the dimension of such subspaces. Sometimes this finiteness condition is replaced by taking $\varphi_1,\varphi_2$ in higher order Sobolev-type spaces, and often an $\epsilon$-loss in the decay appears, as above (see e.g.\ \cite{moore}). On the contrary, in \eqref{coeff} we have the low regularity space $M^1$, and functions in $M^1$ do not need to have any differentiability, even in a fractional sense.
\par
Weak-type estimates for matrix coefficients such as \eqref{coeff} seem of great interest in their own right; for example, they could play a key role in extending Cowling's strengthened version of the Kunze-Stein phenomenon \cite{cowling3} to groups of rank higher than 1.\par

As a consequence of the dispersive estimates we therefore obtain the following Strichartz-type estimates. 
\begin{theorem}[Strichartz estimates]\label{mainthm3} Let $G=Mp(n,\R)$ with the Haar measure. The following estimates hold: 
\[
\|\widehat{S}\psi\|_{L^q(G;M^r)}\lesssim\|\psi\|_{L^2},
\] for 
\[
\frac{4n}{q}+\frac{1}{r}\leq \frac{1}{2},\quad 2\leq q,r\leq \infty.
\]
\end{theorem}
The range of admissible pairs $(q,r)$ in Theorem \ref{mainthm3} is represented in Figure \ref{figura}, which also shows a comparison with the case of the Schr\"odinger group (as already observed, the one-parameter group $e^{it\Delta}$ is a subgroup of $Mp(n,\R)$). Notice however that the exponent $r$ refers to different function spaces; in fact we have $L^r\subset M^r$ for $2\leq r\leq\infty$, with strict inclusion when $r>2$. As one can see, the admissibility condition implies $q\geq 8n$. Also, we have a whole region of admissible pairs, and not just a segment, because the modulation spaces $M^r$ are nested, unlike the Lebesgue spaces. Let us observe that, compared with the trivial estimate for $q=\infty,r=2$, again the other admissible pairs $(q,r)$ represent a gain (loss) of time (space) decay at infinity. Instead, we do no longer have any smoothing effect, as expected: among the metaplectic operators we also meet linear changes of variables, which do not produce smoothing in any reasonable space. This is in turn related to the fact that $M^1\subset M^\infty$ in \eqref{dispersiva}. 

\par
\begin{figure}[htb]\label{figura}
\centering
\begin{tikzpicture}[xscale=0.5,yscale=0.6]  
\draw[-stealth,thick] (0,0) -- (5.5,0) node[above right] {$1/q$};
\draw[-stealth,thick] (0,0) -- (0,4.5) node[right] {$1/r$};
\draw [-, thick]  (0,3) node [left]{$\frac{1}{2}$} -- (4,1);
\draw [-, thick] (0,3) -- (2,0) node [below]{$\frac{1}{8n}$};
\draw [-, dashed] (4,1) -- (4,0) node [below]{$\frac{1}{2}$};
\draw [fill=gray] (0,0) -- (2,0) -- (0,3);
\draw [<-, thick] (2,2) -- (4,3.2) node [right]{Schr\"odinger equation};
\draw [-,dashed] (0,1) node[left]{$\frac{n-2}{2n}$} -- (4,1);
\draw [<-, thick] (1,0.5) -- (3.3,1.8) node [right]{metaplectic representation};
\node at (9,4.5) {(Lebesgue/modulation space exponent)};
\end{tikzpicture}
\caption{Admissible pairs for Strichartz estimates}
\end{figure}
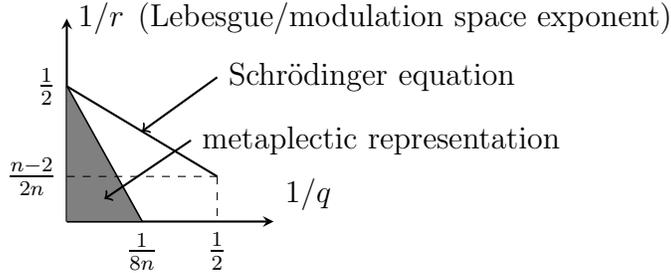

Let us observe that similar estimates seem worth investigating for other unitary representations, e.g.\  the oscillatory representation restricted to subgroups of ${Mp}(n,\R)$ (cf.\ \cite{cordero01,cordero02,cordero03,cordero04}), unitary representations of linear Lie groups such as ${SL}(n,\R)$ or more general semisimple Lie groups, where the Cartan decomposition should play the role of our singular value analysis. Part of the problem is to identify low regularity spaces strictly tailored to the given representation, playing the role of the modulation spaces used here. 
We plan to carry on this investigation in future work. \par\medskip
The paper is organized as follows. In Section \ref{sec2} we recall some preliminary results on time-frequency and symplectic methods used in the proofs of the main results. That material is mainly extracted from \cite{DeGosson}. Section \ref{sec3} is devoted to the proof of the above results. Finally in Section \ref{sec4} we collected some concluding remarks.

\section{Preliminaries}\label{sec2}
We recall here a number of definitions and results that we will use in the following. We refer to \cite{DeGosson,helgason,leray} for details. 

\subsection{Notation} We denote by $\langle\cdot,\cdot\rangle$ the inner product in $L^2(\rn)$, linear in the first argument. The notation $A\lesssim B$, for expressions $A,B\geq0$, means $A\leq C B$ for a constant $C$ depending only on the dimension $n$ and parameters which are fixed in the context. We also write $A\asymp B$ for $A\lesssim B$ and $B\lesssim A$. 

\subsection{The symplectic group}\label{sec2.1}
The symplectic group $Sp(n,\mathbb{R})$ is the group of $2n\times 2n$ real matrices $S$ such that $S^T JS=J$, where 
\[
J=\begin{pmatrix}
0& I\\
-I&0
\end{pmatrix}. 
\]
We recall that every symplectic matrix $S$ admits a unique polar decomposition 
$S=S_0U$
where $S_0$ is symplectic, symmetric and positive definite and $U$ is a symplectic rotation, i.e.\ belongs to 
\[
U(2n,\mathbb{R}):=Sp(n,\R)\cap O(2n,\R)\simeq U(n).
\]
A positive definite matrix can always be diagonalized using an orthogonal matrix. When this matrix is in addition symplectic we can use a symplectic rotation to perform this diagonalization: if $S$ is positive definite, there exists $U\in U(2n,\mathbb{R})$ such that $S=U^TDU$ where \[D={\rm diag}(\lambda_1,\ldots,\lambda_n,\lambda^{-1}_1,\ldots,\lambda^{-1}_n)\] and $\lambda_1\geq \ldots\geq \lambda_n\geq \lambda_n^{-1}\geq \ldots \geq \lambda_1^{-1}>0$ are the eigenvalues of $S$.\par
By combining polar decomposition and this diagonalization result we see that every symplectic matrix $S$ can be written as 
\[
S=U_1 D U_2
\]
with $U_1,U_2\in U(2n,\mathbb{R})$ and $D$ diagonal as above, where $\lambda_1\geq \ldots\geq \lambda_n\geq \lambda_n^{-1}\geq \ldots \geq \lambda_1^{-1}>0$ are now the singular values of $S$. 
\subsection*{Integration on the symplectic group} $Sp(n,\mathbb{R})$ turns out to be a unimodular Lie group. The following integration formula for $U(2n,\mathbb{R})$-bi-invariant functions on $Sp(n,\mathbb{R})$ will be crucial in the following. \par
Recall that $f:Sp(n,\mathbb{R})\to \mathbb{C}$ is called $U(2n,\mathbb{R})$-bi-invariant if $f(U_1 S U_2)=f(S)$ for every $S\in Sp(n,\mathbb{R})$, $U_1,U_2\in U(2n,\R)$.\par
Consider the Abelian subgroup $A=\{a_t\}$ of $Sp(n,\mathbb{R})$ given by
\[
a_t=\left(\begin{matrix}e^{\frac{t}{2}} & 0 \\ 0 & e^{-\frac{t}{2}}\end{matrix}\right),\quad t={\rm diag}(t_1,\ldots,t_n),\ (t_1,\ldots,t_n)\in\rn.
\]
 If $f$ is a $U(2n,\R)$-bi-invariant function on $Sp(n,\mathbb{R})$, its integral with respect to the Haar measure is given by 
 \begin{multline}\label{integrale}
 \int_{Sp(n,\mathbb{R})} f(S)dS\\
 =C\int_{t_1\geq\ldots \geq t_n\geq 0}f(a_t)\prod_{i<j}\sinh\frac{t_i-t_j}{2}\prod_{i\leq j}\sinh\frac{t_i+t_j}{2}\,dt_1\ldots dt_n
 \end{multline}
 for some constant $C>0$.
 \subsection*{Weak-type Young inequality on unimodular groups}
 We will also need the Young inequality for weak type spaces, which reads as follows.\par
On a measure space $X$, for $0< p<\infty$ the weak-type Lebesgue space $L^{p,\infty}(X)$ is defined as the space of measurable functions $f:X\to\mathbb{C}$ such that 
\[
\|f\|_{L^{p,\infty}}:=\sup\limits_{\lambda>0} \{\lambda \cdot ({\rm meas}\{x:\,|f(x)|\geq \lambda\})^{1/p}\}<\infty. 
\]
 Let now $G$ be a unimodular locally compact Hausdorff group. Let  \[ 1<p,q,r<\infty,\quad
 \frac{1}{p}+\frac{1}{r}=\frac{1}{q}+1.
 \]
  Then there exists a constant $C_{p,q,r}>0$ such that for all $f$ in $L^p(G)$ and $g$ in $L^{r,\infty}(G)$ we have 
 \begin{equation}\label{wyoung}
  \|f\ast g\|_{L^{q}(G)}\leq C_{p,q,r}\|g\|_{L^{r,\infty}(G)}\|f\|_{L^p(G)}.
  \end{equation}

\subsection{The metaplectic group}\label{sec2.2}
 There are many construction of the metaplectic group $Mp(n,\mathbb{R})$, i.e.\ the double covering of the symplectic group $Sp(n,\mathbb{R})$, and the metaplectic representation $Mp(n,\mathbb{R})\rightarrow \mathcal{U}(L^2(\rn))$. Since it turns out to be a faithful representation, we can in fact think of group elements as unitary operators themselves. This is the point of view of the following construction, where  $Mp(n,\mathbb{R})$ is defined as a subgroup of the unitary group $\mathcal{U}(L^2(\rn))$ and the corresponding representation is just the inclusion. The difficult point is to prove the existence of a projection 
 \[
 \pi: Mp(n,\mathbb{R})\to Sp(n,\mathbb{R})
 \]
 which makes $Mp(n,\mathbb{R})$ the double covering of $Sp(n,\mathbb{R})$. \par
 We recall here the main points of the construction, and we refer to \cite{DeGosson} and \cite{leray} for details.  \par
It can be proved that the symplectic group $Sp(n,\mathbb{R})$ is generated by the so-called free symplectic matrices \[S=\left(\begin{matrix}A & B \\ C & D\end{matrix}\right)\in Sp(n,\mathbb{R}),\quad \det B\neq 0.\] To each such matrix we associate the generating function \[W(x,x')=\frac{1}{2}DB^{-1}x\cdot x-B^{-1}x\cdot x'+\frac{1}{2}B^{-1}Ax'\cdot x'.\] Conversely, to every polynomial of the type \[W(x,x')=\frac{1}{2}Px\cdot x-Lx\cdot x'+\frac{1}{2}Qx'\cdot x'\] with \[P=P^T, Q=Q^T\] and \[\det L\neq 0\] we can associate a free symplectic matrix, namely \[S_W=\left(\begin{matrix}L^{-1}Q & L^{-1} \\ PL^{-1}Q-L^T & PL^{-1}\end{matrix}\right).\] 
Now, given $S_W$ as above and $m\in\mathbb{Z}$ such that 
\[
m\pi \equiv {\rm arg}\,\det L\quad {\rm mod}\, 2\pi,
\]
we define the operator $\widehat{S}_{W,m}$ by setting, for $\psi \in \mathcal{S}(\mathbb{R}^n)$,
\[\widehat{S}_{W,m}\psi(x)=\frac{1}{(2\pi i)^{n/2}}\Delta(W)\int_{\mathbb{R}^n}e^{iW(x,x')}\psi(x')dx'\] 
(to be clear, $(2\pi i)^{n/2}=(2\pi)^{n/2}e^{i\pi n/4}$) where
\[\Delta(W)=i^m\sqrt{|\det L|}.\]
The operator $\widehat{S}_{W,m}$ is called a {\it quadratic Fourier transform} associated to the free symplectic matrix $S_W$.
 The class modulo 4 of the integer $m$ is called {\it Maslov index} of $\widehat{S}_{W,m}$. Observe that if $m$ is one choice of Maslov index, then $m+2$ is another equally good choice: hence to each function $W$ we associate two operators, namely $\widehat{S}_{W,m}$ and $\widehat{S}_{W,m+2}=-\widehat{S}_{W,m}$.

 The quadratic Fourier transform corresponding to the choices $S_W=J$ and $m=0$ is denoted by $\widehat{J}$. The generating function of $J$ being simply $W(x,x')=-x\cdot x'$, it follows that
\[\widehat{J}\psi(x)=\frac{1}{(2\pi i)^{n/2}}\int_{\mathbb{R}^n}e^{-ix\cdot x'}\psi(x')dx'=\frac{1}{i^{n/2}}F\psi(x)
\] for $\psi\in \mathcal{S}(\mathbb{R}^n)$, where $F$ is the usual unitary Fourier transform.\par
The quadratic Fourier transforms $\widehat{S}_{W,m}$ form a subset of the group $\mathcal{U}(L^2(\mathbb{R}^n))$ of unitary operators acting on $L^2(\mathbb{R}^n)$, which is closed under the operation of inversion and they generate a subgroup of $\mathcal{U}(L^2(\mathbb{R}^n))$ which is, by definition,
the metaplectic group $Mp(n,\mathbb{R})$. The elements of $Mp(n,\mathbb{R})$ are called metaplectic operators.
\par
Every $\widehat{S}\in Mp(n,\mathbb{R})$ is thus, by definition, a product \[\widehat{S}_{W_1,m_1}\ldots\widehat{S}_{W_k,m_k}\] of metaplectic operators associated to free symplectic matrices.\par
In fact, it can be proved that every $\widehat{S}\in Mp(n,\mathbb{R})$ can be written as a product of exactly two quadratic Fourier transforms: $\widehat{S}=\widehat{S}_{W,m}\widehat{S}_{W',m'}$. 

Now, it can be proved that the map 
\[
\widehat{S}_{W,m}\longmapsto S_W
\]
extends to a group homomorphism 
\[
\pi: Mp(n,\mathbb{R})\to Sp(n,\mathbb{R}),
\]
which is in fact a double covering. \par
We also observe that each metaplectic operator is, by construction, a unitary operator in $L^2(\rn)$, but also an authomorphism of $\mathcal{S}(\rn)$ and of $\mathcal{S}'(\rn)$. 

\subsection{Modulation spaces}\label{sec2.3}
\par
Fix a window function $\varphi\in \mathcal{S}(\mathbb{R}^n)\setminus\{0\}$. The short-time Fourier transform (STFT) of a function/temperate distribution $\psi\in \mathcal{S'}(\mathbb{R}^n)$ with respect to $\varphi$ is defined by
\[V_\varphi \psi(x,\xi)=(2\pi)^{-n}\int_{\mathbb{R}^n}e^{-i\xi \cdot y}\psi(y)\overline{\varphi(y-x)}dy, \quad x,\xi \in \mathbb{R}^n.
\]
For $1\leq p, q\leq \infty$ and a Schwartz function $\varphi\in \mathcal{S}(\mathbb{R}^n)\setminus\{0\}$, the modulation space $M^{p,q}(\mathbb{R}^n)$ is defined as the space of $\psi\in \mathcal{S'}(\mathbb{R}^n)$ such that 
\[\|\psi\|_{M^{p,q}}:=\Big(\int_{\mathbb{R}^{n}}\Big(\int_{\mathbb{R}^{n}}|V_\varphi \psi(x,\xi)|^p dx\Big)^{q/p} d\xi \Big)^{1/q} < \infty,
\]
with obvious changes if $p=\infty$ or $q=\infty$. \par
If $p=q$, then we write $M^{p}$ instead of $M^{p,p}$.\par
We will also need a variant, sometimes called Wiener amalgam norm in the literature, defined by
\[
\|\psi\|_{W(FL^p, L^q)}:=(\int_{\mathbb{R}^{n}}\Big(\int_{\mathbb{R}^{n}}|V_\varphi \psi(x,\xi)|^p d\xi\Big)^{q/p} dx \Big)^{1/q},
\]
where the Lebesgue norms appear in the inverse order. Both these norms provide a measure of the time-frequency concentration of a function and are widely used in Time-frequency Analysis \cite{DeGosson,Grochenig}. \par
We have $M^{p_1,q_1}\subseteq M^{p_2,q_2}$ if and only if $p_1\leq p_2$ and $q_1\leq q_2$. Similarly $W(F L^{p_1},L^{q_1})\subseteq W(F L^{p_2},L^{q_2})$ if and only if $p_1\leq p_2$ and $q_1\leq q_2$.\par
The duality goes as expected: 
\[
(M^{p,q})'= M^{p',q'},\quad 1\leq p,q<\infty,
\]
and in particular 
\begin{equation}\label{dualita}
|\langle f,g\rangle|\lesssim \|f\|_{M^p}\|g\|_{M^{p'}}.
\end{equation}
In the dispersive estimates we meet, in particular, the Gelfand triple
\[
M^1\subset L^2(\rn)\subset M^\infty.
\]
 We observe that 
\[
\mathcal{S}(\rn)\subset M^1\subset L^2(\rn)
\]
with dense and strict inclusions. For atomic characterizations of the space $M^1$ we refer to \cite{DeGosson,Grochenig}. \par
We will also use the complex interpolation theory for modulation spaces, which reads as follows:
for $1\leq p,q,p_i,q_i\leq \infty$, $i=0,1$, $0\leq \theta\leq 1$, \[
\frac{1}{p}=\frac{1-\vartheta}{p_0}+\frac{\vartheta}{p_1},\quad \frac{1}{q}=\frac{1-\vartheta}{q_0}+\frac{\vartheta}{q_1},\]
 we have 
\[(M^{p_0,q_0},M^{p_1,q_1})_{\vartheta}=M^{p,q}.
\]

\subsection{The Wigner distribution}\label{sec2.4}
We now introduce a quadratic time-frequency distribution which will play a key role in the following. Again it represents a basic tool in the analysis of signals \cite{Grochenig} and in phase space Quantum Mechanics \cite{DeGosson,DeGosson2}. We refer to \cite{DeGosson,DeGosson2} for details. \par
The cross-Wigner distribution $W(\psi,\varphi)$ of functions $\psi,\varphi\in L^2(\mathbb{R}^n)$ is defined to be
\[
W(\psi,\varphi)(x,\xi)=(2\pi)^{-n}\int_{\mathbb{R}^n}e^{- i\xi \cdot y}\psi\Big(x+\frac{y}{2}\Big)\overline{\varphi\Big(x-\frac{y}{2}\Big)}dy.
\]
We also set $W\psi=W(\psi,\psi)$. \par
We recall the important {\it Moyal identity} (see e.g.\ \cite[Theorem 182]{DeGosson}):
\begin{equation}\label{moyal}
\langle W\psi,W\varphi\rangle_{L^2(\mathbb{R}^{2n})}=(2\pi)^{-n} |\langle \psi,\varphi\rangle|^2. 
\end{equation}
We will also need the following estimates.
\begin{proposition}\label{pro0}
We have
\begin{equation}\label{eq0}
\|W(\psi,\varphi)\|_{L^1(\rnn)}\lesssim \|\varphi\|_{M^1} \|\psi\|_{M^1},
\end{equation}
\begin{equation}\label{eq1}
\int_{\rn} \sup_{x\in\rn} |W(\psi,\varphi)(x,\xi)|d\xi\lesssim \|\varphi\|_{M^1}\|\psi\|_{M^{\infty,1}}
\end{equation}
and 
\begin{equation}\label{eq2}
\int_{\rn} \sup_{\xi\in\rn} |W(\psi,\varphi)(x,\xi)|dx\lesssim \|\varphi\|_{M^1}\|\psi\|_{W(F L^\infty, L^1)}.
\end{equation}
\end{proposition}
\begin{proof}
Formula \eqref{eq0} is proved in \cite[Proposition 3.6.5]{DeGosson}.\par
Let us prove \eqref{eq1} and \eqref{eq2}.
It is easy to see that
\begin{equation}\label{eq00}
|W(\psi,\varphi)(x,\xi)|=2^n|V_\varphi \psi(2x,2\xi)|
\end{equation}
so that it is sufficient to prove similar estimates with $W(\psi,\varphi)(x,\xi)$ replaced by $V_\varphi \psi(x,\xi)$. To this end we recall from \cite[Lemma 11.3.3]{Grochenig} that, for $\varphi,\varphi_0 \in \mathcal{S}(\mathbb{R}^n)$ such that $\|\varphi_0\| \neq 0$ and $\psi\in \mathcal{S'}(\mathbb{R}^n)$ we have 
\begin{equation}\label{conv}
 |V_{\varphi}\psi(x,\xi)|\lesssim \frac{1}{\|\varphi_0\|^2}(|V_{\varphi_0} \psi|*|V_{\varphi}\varphi_0 |)(x,\xi),
\end{equation}
 for all $(x,\xi)\in \mathbb{R}^{2n}$.

Now, we apply this inequality with a fixed Schwartz window $\varphi_0$ and we also observe that 
\[
|V_{\varphi}\varphi_0(x,\xi)|=|V_{\varphi_0}\varphi(-x,-p)|.
\]
The desired estimates for $V_{\varphi}\psi(x,\xi)$ then follow by applying the Young inequality for mixed-norm Lebesgue spaces in \eqref{conv}. 
\end{proof}
One of the most important property of the cross--Wigner distribution is its covariance with respect to the action of metaplectic operators. In fact we have (see e.g.\ \cite[Corollary 2.17]{DeGosson})
\begin{equation}\label{covarianza}
W(\widehat{S}\psi, \widehat{S}\varphi)(z)= W(\psi,\varphi)(S^{-1}z),\quad z\in\rn\times\rn. 
\end{equation}
for every $\widehat{S}\in Mp(n,\mathbb{R})$, with projection ${S}\in Sp(n,\mathbb{R})$.

\section{Proof of the main results}\label{sec3}
In this section we prove Theorems \ref{mainthm}, Corollary \ref{mainthm2} and Theorem \ref{mainthm3}.
\begin{proof}[Proof of Theorem \ref{mainthm}]
By duality it is equivalent to prove that \[| \langle \widehat{S}\psi,\varphi\rangle|\lesssim (\lambda_1(S)\ldots\lambda_n(S))^{-1/2}
\|\psi\|_{M^1}\|\varphi\|_{M^1}.\] By the Moyal identity \eqref{moyal} and the covariance property \eqref{covarianza}, we have 
\begin{align*}
|\langle \widehat{S}\psi,\varphi\rangle|^2&=(2\pi)^n \langle W(\widehat{S}\psi),W\varphi\rangle_{L^2(\mathbb{R}^{2n})}
\\ 
&=(2\pi)^n\langle W\psi(S^{-1}\cdot),W\varphi\rangle_{L^2(\mathbb{R}^{2n})}.
\end{align*}
We now can write $S^{-1}=S_1U_1$ with $S_1\in Sp(n,\mathbb{R})$ positive definite and $U_1\in U(2n,\mathbb{R})$. Hence, by an orthogonal change of variable we obtain
\[
|\langle \widehat{S}\psi,\varphi\rangle|^2
=(2\pi)^n \langle W\psi(S_1\cdot),W\varphi(U_1^T\cdot)\rangle_{L^2(\rnn)}.\] We now diagonalize $S_1$, $S_1=U_2^TDU_2$ (see Section \ref{sec2.1}) where
\[
D={\rm diag}(\lambda_1,\ldots,\lambda_n,\lambda^{-1}_1,\ldots,\lambda^{-1}_n)
\]
with $\lambda_1\geq \ldots\geq \lambda_n\geq \lambda_n^{-1}\geq \ldots \geq \lambda_1^{-1}>0$ and $U_2\in U(2n,\mathbb{R})$. With a further change of variable we obtain\[
|\langle \widehat{S}\psi,\varphi\rangle|^2
=(2\pi)^n\langle W\psi(U_2^TD\,\cdot),W\varphi(U_1^TU_2^T\cdot)\rangle _{L^2(\mathbb{R}^{2n})}.
\] 
Let 
\[
F_1=W\psi(U_2^T\cdot)=W(\widehat{U}_2\psi),
\] 
and
 \[
F_2=W\varphi(U_1^TU_2^T\cdot)=W(\widehat{U}_2\widehat{U}_1\varphi).
\]
 We estimate 
 \begin{align*}
 \langle &W\psi(U_2^TD\,\cdot),W\varphi(U_1^TU_2^T\cdot)\rangle_{L^2({\mathbb{R}^{2n}})}
 \\ &=
\int_{\mathbb{R}^{2n}} F_1(\lambda_1 x_1,\ldots,\lambda_n x_n,\lambda_1^{-1}\xi_1,\ldots,\lambda_n^{-1}\xi_n) \overline{F_2(x,\xi)} dxd\xi
\\
&\leq \int_{\mathbb{R}^{2n}} \sup_{\xi\in \mathbb{R}^n}|F_1(\lambda_1x_1,\ldots,\lambda_n x_n,\xi_1,\ldots,\xi_n)|\sup_{x\in \mathbb{R}^n}|F_2(x,\xi)|dxd\xi
\\
&=\lambda_1^{-1}\ldots \lambda_n^{-1} \int_{\mathbb{R}^{n}} \sup_{\xi\in \mathbb{R}^n}|F_1(x,\xi)|dx \int_{\mathbb{R}^n}\sup_{x\in \mathbb{R}^n}|F_2(x,\xi)|d\xi\leq
\\
 &\lesssim \lambda_1^{-1}\ldots \lambda_n^{-1}\|\widehat{U}_2\psi\|_{M^1} \|\widehat{U}_2\psi\|_{ W(\mathcal{F}L^\infty,L^1)   } \|\widehat{U}_2\widehat{U}_1\varphi\|_{M^1}\|\widehat{U}_2\widehat{U}_1\varphi\|_{M^{\infty,1} }
 ,
  \end{align*} where we used, in the last line, Proposition \ref{pro0}.\par
  Using the inclusions 
  \[
  M^1=M^{1,1}\hookrightarrow M^{\infty,1},\quad M^1=W(\mathcal{F}L^1,L^1) \hookrightarrow W(\mathcal{F}L^\infty,L^1)\]
   we continue the above estimate as 
   \[\lesssim \lambda_1^{-1}\ldots \lambda_n^{-1}\|\widehat{U}_2\psi\|_{M^1}^2\|\widehat{U}_2\widehat{U}_1\varphi\|_{M^1}^2.
   \]
It is then sufficient to show that
\begin{equation}\label{3.2} \|\widehat{U}_2\psi\|_{M^1}\leq C\|\psi\|_{M^1}
\end{equation}
and
\begin{equation}\label{3.1}\|\widehat{U}_2\widehat{U}_1\varphi\|_{M^1}\leq C \|\varphi\|_{M^1}
\end{equation} for a constant $C>0$ independent of $\widehat{U}_1$, $\widehat{U}_2$.\par
Let us verify \eqref{3.2}, which implies \eqref{3.1} too.
 By the definition of the $M^1$ norm and \eqref{eq00} we have \[\|\widehat{U}_2\psi\|_{M^1}\asymp \|W(\widehat{U}_2\psi,\varphi)\|_{L^1(\rnn)}\] for some fixed $\varphi\in \mathcal{S}(\mathbb{R}^n)\setminus\{0\}$, which by covariance is equal to 
 \[\|W(\psi,\widehat{U}_2^{-1}\varphi)(U_2^{-1}\cdot)\|_{L^1(\rnn)}=\|W(\psi,\widehat{U}_2^{-1}\varphi)\|_{L^1(\rnn)}.\] 
Hence, using \eqref{eq0} and the continuous embedding $\mathcal{S}(\mathbb{R}^n)\hookrightarrow M^1$ it is sufficient to prove that if $\varphi\in \mathcal{S}(\mathbb{R}^n)$ then \[\{\hat{U}\varphi:\,U\in U(2n,\mathbb{R})\}\] is a bounded subset of $\mathcal{S}(\mathbb{R}^n)$, that is,  every Schwartz seminorm is bounded on it. Since $U(2n,\mathbb{R})$ is compact it is sufficient to show that every seminorm is locally bounded, i.e.\ we can limit ourselves to take $U$ in a sufficiently small neighborhood for any fixed $U_0\in U(2n,\mathbb{R})$. Equivalently, we can consider $U$ of the form $U=U_1J^{-1}U_0$ where $U_1$ belongs to a sufficiently small neighborhood of $J$ in $U(2n,\R)$. Now, 
\begin{align*}
\hat{U}\varphi(x)&=\pm \widehat{U}_1[\widehat{J}^{-1}\widehat{U}_0\varphi](x)\\
&=c \sqrt{|\det L|}\int_{\mathbb{R}^n}e^{\frac{i}{2}Px\cdot x-iLx\cdot y+\frac{i}{2}Qy\cdot y}[\underbrace{\widehat{J}^{-1}\widehat{U}_0\varphi]}_{\in \mathcal{S}(\mathbb{R}^n)}(y)dy
\end{align*}
 where $|c|=1$ and, say, $\|P\|<\epsilon$, $\|Q\|<\epsilon$, $\|L-I\|<\epsilon$. If $\epsilon<1$, it is clear that $\hat{U}\varphi$ belongs to a bounded subset of $\mathcal{S}(\mathbb{R}^n)$, as one can verify by direct inspection.
\end{proof}
In order to prove Corollary \ref{mainthm2} we need the following preliminary result. 
\begin{proposition}\label{pro1}
Let $\alpha>0$, $\beta>0$. Consider the function 
\[
h(S)=(\lambda_1(S)\ldots\lambda_n(S))^{-\alpha}
\]
on $Sp(n,\R)$, where $\lambda_1(S)\ldots\lambda_n(S)$ are the singular values $\geq1$ of the symplectic matrix $S$.\par
 We have $h\in L^{\beta,\infty}$ on $Sp(n,\R)$, with respect to the Haar measure, if
\[
 \alpha\beta\geq 2n.
 \]
\end{proposition}
\begin{proof}
We have to estimate the measure of the set 
\[
\mathcal{D}_\lambda=\{S\in Sp(n,\R):\ h(S)\geq \lambda\}, \quad \lambda>0
\]
or equivalently 
\[
\int_{Sp(n,\R)} \chi_{\mathcal{D}_\lambda} dS,
\]
where $\chi_{\mathcal{D}_\lambda}$ is the indicator function of $\mathcal{D}_\lambda$. Observe that $\mathcal{D}_\lambda=\emptyset$ if $\lambda>1$ so that we can suppose $0<\lambda\leq 1$. \par
We use formula \eqref{integrale} with $f=\chi_{\mathcal{D}_\lambda}$ since $h$, and therefore $f$, is $U(2n,\R)$-bi-invariant. With the notation in \eqref{integrale} we have 
\[
h(a_t)=e^{-\alpha(t_1+\ldots +t_n)/2},
\]
where $t=(t_1,\ldots,t_n)$ and $t_1, \ldots, t_n\geq0$. Hence $h(a_t)\geq\lambda$ if and only if 
\[
t_1+t_2+\ldots+t_n\leq A_\lambda:=-2\log \lambda/\alpha.
\]
By \eqref{integrale},
\[
{\rm meas}\,\mathcal{D}_\lambda=C
\int_{t_1\geq t_2\geq...\geq t_n\geq 0\atop t_1+...+t_n\leq A_\lambda} \prod_{i<j} \sinh\frac{t_i-t_j}{2}\, \prod_{i\leq j} \sinh\frac{t_i+t_j}{2}\, dt_n\ldots dt_1.\] 

Now we have
\begin{align*}
\prod_{i<j} \sinh &\frac{t_i-t_j}{2}\, \prod_{i\leq j} \sinh\frac{t_i+t_j}{2}\\
&\leq \exp\Big({\sum_{i<j}\Big(\frac{t_i-t_j}{2}+\frac{t_i+t_j}{2}\Big)+t_1+...+t_n}\Big)\\
&=\exp\big({(n-1)t_1+(n-2)t_2+...+t_{n-1}+t_1+...+t_n}\big)\\
&=e^{t_n}e^{2t_{n-1}}e^{3t_{n-2}}\ldots e^{nt_1}.
\end{align*}
By first integrating with respect to the variable $t_n$ from $t_n=0$ to $t_n=A_\lambda-t_{n-1}-\ldots-t_1$, we obtain
\[
{\rm meas}\,\mathcal{D}_\lambda\leq C \int_{t_1\geq t_2\geq...\geq t_{n-1}\geq 0\atop t_{n-1}\leq A_\lambda-t_{n-2}-...-t_1}e^{A_\lambda}e^{t_{n-1}}\ldots e^{(n-1)t_1}\, dt_{n-1}\ldots dt_1.
\]
Now we can repeat the same argument for $t_{n-1}$ and so on. We obtain
\[
{\rm meas}\,\mathcal{D}_\lambda\leq Ce^{nA_\lambda}=C \lambda^{-2n/\alpha},\quad 0<\lambda\leq 1.
\]
Hence 
${\rm meas}\,\mathcal{D}_\lambda\leq C'\lambda^{-\beta}$ if $2n/\alpha\leq \beta$, which is the desired result. 
\end{proof}

\begin{proof}[Proof of Corollary \ref{mainthm2}]
Using \eqref{dualita} and \eqref{dispersiva} we have 
 \[|\langle \widehat{S}\varphi_1,\varphi_2\rangle|\lesssim (\lambda_1(S)\ldots\lambda_n(S))^{-1/2}\|\varphi_1\|_{M^1}\|\varphi_2\|_{M^1}.
 \]
 Hence it is sufficient to prove that the function 
 \[
 \widehat{S}\longmapsto (\lambda_1(S)\ldots\lambda_n(S))^{-1/2}
 \]
 is in $L^{4n,\infty}$ on $Mp(n,\R)$ with respect to the Haar measure. Since this function factorizes through $Sp(n,\R)$ it is enough to prove that the function 
 \[
 h(S):= (\lambda_1(S)\ldots\lambda_n(S))^{-1/2}
 \]
 is in $L^{4n,\infty}$ on $Sp(n,\R)$. This follows from Proposition \ref{pro1} with $\alpha=1/2$ and $\beta=4n$.  
\end{proof}

We are now ready to prove the Strichartz estimates for the metaplectic representation.  
\begin{proof}[Proof of Theorem \ref{mainthm3}]
We know that 
\begin{equation}\label{elle2}
\|\widehat{S}\psi\|_{L^2}=\|\psi\|_{L^2}
\end{equation} 
for $\psi\in L^2(\rn)$, which gives the desired Strichartz estimate for 
$q=\infty,\ r=2$, because $M^2=L^2$, and also for $q=\infty,\ 2\leq r\leq\infty$, because $L^2\hookrightarrow M^r$ for $r\geq 2$. Hence from now on we can suppose 
$
q<\infty.
$
\par
Now by Theorem \ref{mainthm}, 
\[\|\widehat{S}\psi\|_{M^\infty}\lesssim (\lambda_1(S)\ldots\lambda_n(S))^{-1/2}\|\psi\|_{M^1}.\]
By interpolation with \eqref{elle2} we obtain, for every $2\leq r\leq\infty$,
\begin{equation}\label{interp}
\|\widehat{S}\psi \|_{M^r}\lesssim(\lambda_1(S) \ldots\lambda_n(S))^{-(\frac{1}{2}-\frac{1}{r})}\|\psi\|_{M^{r'}}.
\end{equation}
Let now $G=Mp(n,\R)$, as in the statement. We apply the usual $ TT^\ast$ method (see \cite[page 75]{Tao}) to the operator $T\psi=\widehat{S}\psi$. To prove that \[
T:L^2\to L^q(G;M^r)\]
 continuously, we will verify that \[
TT^\ast :L^{q'}(G;M^{r'})\to L^q(G;M^r)
\]
 continuously. \par
 We have 

\[T^*F(\cdot)=\int_G\widehat{S}^{-1}F(\widehat{S},\cdot)d\widehat{S}
\]
if $F(\widehat{S},x)$ is, say, a continuos function on $G\times \rn$ with compact support. 

Hence
\[[TT^*F](\widehat{S},\cdot)=\int_G \widehat{S}\widehat{S'}^{-1}F(\widehat{S'},\cdot)d\widehat{S'}.
\]
 Now, using \eqref{interp} we can estimate this expression, for every $2\leq r\leq\infty$, $1\leq q\leq\infty$, as
\begin{align}
\|TT^*F\|_{L^q(G;M^r)}&\leq \|\int_G \|\widehat{S}\widehat{S'}^{-1}F(\widehat{S'},\cdot)\|_{M^r}d\widehat{S'}\|_{L^q(G)}\nonumber\\
&\leq \|\int_G (h \circ \pi)(\widehat{S}\widehat{S'}^{-1})\|F(\widehat{S'},\cdot)\|_{M^{r'}}d\widehat{S'}\|_{L^q(G)},\label{ultima}
\end{align}
where we set
\[
h(S)=(\lambda_1(S) \ldots\lambda_n(S))^{-(\frac{1}{2}-\frac{1}{r})}
\]
as a function on $Sp(n,\R)$ and $\pi:G=Mp(n,\R)\to Sp(n,\R)$ is the projection. \par
Now suppose that the pair $(q,r)$ satisfies $2\leq q,r\leq\infty$ and $\frac{4n}{q}+\frac{1}{r}\leq \frac{1}{2}$; see Figure \ref{figura}. Observe that this implies $q>2$ and we are also supposing $q<\infty$, which implies $r>2$.  Choose \[
\alpha=\frac{1}{2}-\frac{1}{r},\quad \frac{1}{\beta}= \frac{2}{q}.
\]
 We see that $\alpha,\beta>0$ and $\alpha\beta\geq 2n$ so that by Proposition \ref{pro1} we have $h\in L^{\beta,\infty}$ in $Sp(n,\R)$ and $h \circ \pi\in L^{\beta,\infty}$ on $G$. We moreover have 
\[
\frac{1}{q}+1=\frac{1}{q'}+\frac{1}{\beta},\quad 1<q,q',\beta<\infty.
\]
Hence we can apply the weak-type Young inequality \eqref{wyoung} on $G$ to the last expression in \eqref{ultima}, and we see that it is therefore dominated by 
$
\|F\|_{L^{q'}(G;M^{r'})}.
$\par
This concludes the proof. 
\end{proof}
\section{Concluding remarks}\label{sec4}
\subsection{The motivation for modulation spaces}
Let us point out the main elements which led us to consider the modulation space $M^1$ and its dual $M^\infty$ as natural candidates for the dispersive estimate \eqref{dispersiva}. \par
Estimate \eqref{dispersiva} clearly does not hold with $M^1$ and $M^\infty$ replaced by $L^1$ and $L^\infty$, respectively, because, for example, the pointwise multiplication by $e^{it|x|^2}$ is a metaplectic operator but Lebesgue norms do not detect any decay as $|t|\to+\infty$. Hence we focused on a space which controls $L^1$ decay in space {\it and} $L^1$ decay in momentum, as $M^1$ does indeed.\par
 But in the course of the proof of Theorem \ref{mainthm} we also used in an essential way another property of $M^1$, namely that the
 set of operators $\widehat{U}$ are uniformly bounded on $M^1$ when $U=\pi(\widehat{U})$ varies in $U(2n,\R)$, as proved in \eqref{3.2}. \par
Motivated by these issues, it would be very interesting to get characterizations of function spaces, in particular modulation spaces, in terms of symplectic invariance. 
\subsection{Sharpness of the results} 
It is easy to see that the exponent $-1/2$ in \eqref{dispersiva} is sharp. In fact, one can apply that estimate to a Gaussian function $\psi$ and the metaplectic operator $\widehat{S}\psi(x)=c \sqrt{\det L}\,\psi (Lx)$ (for suitable $c\in\mathbb{C}$, $|c|=1$), with $L={\rm diag}(\lambda_1,\ldots,\lambda_n)$, $\lambda_1,\ldots,\lambda_n\geq 1$. We have $S=(\lambda_1^{-1},\ldots,\lambda_n^{-1},\lambda_1,\ldots,\lambda_n)$ (cf.\ \cite[Proposition 116]{DeGosson}) and
\[
\|\widehat{S}\psi\|_{M^\infty}\asymp (\lambda_1\ldots\lambda_n)^{-1/2},
\]
as proved in \cite[Lemma 3.2]{cordero2} (and in \cite[Lemma 1.8]{toft} in the case $\lambda_1=\ldots=\lambda_n$). \par
Let us observe that the exponent $4n$ in \eqref{coeff} is sharp as well; in fact Howe \cite{howe} proved that for fixed $\varphi_1,\varphi_2\in \mathcal{S}(\mathbb{R}^n)$ the matrix coefficients in general do not belong to $L^{4n}$.

\section*{Acknowledgments}
The authors are very indebted to Professors Elena Cordero, Michael Cowling, Jaques Faraut and Vladimir Uspenskiy for discussions and remarks which improved the paper in an essential way.

\end{document}